\documentclass[oneside,aip,amsmath,amssymb,reprint]{amsart}
\usepackage{mathptmx}
\usepackage[T1]{fontenc}
\usepackage[utf8]{inputenc}
\usepackage{geometry}
\usepackage{amsthm}
\usepackage{amssymb}
\usepackage{stmaryrd}
\usepackage{hyperref}

\makeatletter
\numberwithin{equation}{section}
\numberwithin{figure}{section}
\theoremstyle{plain}
\newtheorem{thm}{Theorem}
\theoremstyle{plain}
\newtheorem{prop}[thm]{Proposition}
\theoremstyle{plain}
\newtheorem{lem}[thm]{Lemma}
\theoremstyle{plain}
\newtheorem{cor}[thm]{Corollary}

\theoremstyle{definition}
\newtheorem{defn}[thm]{Definition}
\newtheorem{rem}[thm]{Remark}

\newcommand{\Cdub}{\left[ \mathbb{C} \right]}
\newcommand{\ree}{\operatorname{Re}}
\newcommand{\imm}{\operatorname{Im}}
\newcommand{\diag}{\operatorname{diag}}
\newcommand{\sgn}{\operatorname{sgn}}
\newcommand{\const}{\operatorname{const.}}
\newcommand{\lr}[1]{\left\llbracket #1 \right\rrbracket}

\makeatother

\begin{document}
\title{On Isomonodromic Deformation of Massive Ising Spinors}
\author{S. C. Park}
\email{scpark@kias.re.kr}
\address{School of Mathematics, Korea Institute for Advanced Study\\
 85 Hoegi-ro, Dongdaemun-gu, Seoul 02455, Republic of Korea }
\begin{abstract}
In this short note, we give a self-contained derivation of the formula
for the $2$-point full-plane Ising spin correlation function under
massive scaling limit in terms of a third Painlevé transcendant.
This formula, first derived in a celebrated work of Wu, McCoy, Tracy,
and Barouch, was subsequently reformulated in terms of the theory
of isomonodromic deformation by Sato, Miwa, and Jimbo. In view of
recent developments in the discrete analysis which have enabled, in
particular, a convergence proof of spin correlation functions on isoradial
lattice, we give a concise and rigorous account of the continuous
theory in the same framework yielding this iconic result. 
\end{abstract}

\maketitle

\section{Introduction}

The Ising model is a lattice model of ferromagnetism renowned for the simplicity in its definition and the mathematical sophistication in its analysis. After its introduction and analysis in one dimension \cite{lenz, ising}, the model found its perhaps the most prolific setting in the self-dual planar lattice $\mathbb{Z}^2$, in part thanks to the model's inherent duality \cite{kramers-wannier} and Onsager's exact computation of the free energy per site \cite{onsager} which demonstrated rigorously that the model undergoes a phase transition at a non-trivial interation strength. Many other formulae regarding fundamental properties of the model soon followed, notably expectations of the spin (\emph{spontaneous magnetization}) and the product of two neighboring spins (\emph{energy density}) \cite{kaufman-onsager-i, yang} (see \cite{mccoy-wu} for a comprehensive overview), and the model went on to be well known for its amenability to a wide variety of analytic methods.

A natural way of studying the large-scale behavior of lattice models is through {scaling limit}, where the lattice scale (mesh size) is taken to zero in fixed space. The critical Ising model in two dimensions was among models conjectured to show emergent \emph{conformal invariance} \cite{BPZ, cardy-i} in scaling limit, whose rigorous proof remained open for decades. However for the \emph{massive} scaling limit, where interaction strength is scaled towards criticality as a function of mesh size, Wu, McCoy, Tracy, and Barouch (WMTB) \cite{wmtb} found already in 1976 a remarkable formula for spin-spin correlations in scaling limit (the 2-point \emph{scaling functions}) in the full-plane, in terms of a third Painlev\'e transcendent and only dependent on the distance between spins (that is, rotationally invariant). Soon after, there was yet another impressive development, with Sato, Miwa, and Jimbo (SMJ) \cite{sato-miwa-jimbo, sato-miwa-jimbo-ii, miwa-jimbo} deriving the formula in terms of \emph{isomonodromic deformation} and generalizing it to the $n$-point case (see also \cite{kako80, patr, palmer}). By comparison, the critical 2-point scaling function was given a rigorous proof only in 2012 by Pinson \cite{pinson}.

Recently, study of the scaling limit of the Ising model has seen a breakthrough with Smirnov's definition of \emph{s-holomorphicity} and corresponding \emph{Riemann-Hilbert boundary value problems} \cite{smirnov-i, smirnov-ii}, which builds on a complex analytic interpretation of discrete relations for \emph{fermion} correlations (e.g. \cite{kadanoff-ceva, perk, mercat}). This approach has particularly proved to be suitable for analysis of the model on generic domains in the complex plane and has facilitated e.g. convergence of critical correlations and interfaces to conformally invariant limits (\cite{hosm2013, chelkak-hongler, chi21, chelkak-duminil-copin-hongler-kemppainen-smirnov, free-interface}, to name a select few). In particular, establishing and characterizing the scaling limit in the full-plane is now an easy corollary of the bounded domain convergence \cite{chelkak-hongler} (see also \cite{chelkak-ems, zigzag}, where the diagonal spin-spin scaling behavior computed in \cite{wu} was re-derived through discrete complex analytic methods).

In addition, this discrete complex analytic approach was noted from early on for its suitability beyond the critical square lattice setup. We are mainly interested in two possible directions of generalization: the discrete holomorphic and harmonic interpretation of critical discrete relations is naturally perturbed to massive counterparts \cite{bedc, hkz, massive-fk-square}, as well as to \emph{isoradial} lattices, where each face is circumscribed by a circle of a fixed radius \cite{chsm2012, dlm}. Exploiting them, this author proved convergence of massive spin correlations on bounded simply connected domains discretized by a square lattice in his Ph.D. thesis \cite{par19} (parts of which we incorporate in this article), which was generalized into the isoradial setup (along with the critical case) by Chelkak, Izyurov, and Mahfouf \cite{cim21}. A crucial element throughout has been developing strategies to work with the particular Riemann-Hilbert boundary condition on general boundaries, regularity assumptions for which have now been largely dropped \cite{par21, cpw}. The pursuit of \emph{universality} across lattice settings has also been recently broadened far beyond the isoradial class by Chelkak's \emph{s-embeddings} \cite{Che18, s-emb}.

We aim here to give an account of the isomonodromic deformation argument given the full isoradial generality of the discrete complex analytic approach. As such, we do not claim originality for this {core mathematical process}, which takes place in the continuum; indeed, the novelties in \cite{par19, cim21} regarding the full-plane model are mainly grounded in their discrete formalism, handled through s-holomorphic discrete spinors converging to continuum objects, the $\mathbb{R}$-linear spinors $f_j$ (Definition \ref{defn:spinor}). While not entirely trivial (see also \cite{cpw}), one can identify them with $\mathbb{C}$-linear objects studied in \cite{sato-miwa-jimbo, kako80, patr}, and follow their analysis. It should also be mentioned that analysis of this particular solution of the Painlev\'e III equation is a substantial problem in its own right, see \cite{mctw} and e.g. \cite[Chapter 15]{fokas}.

Our interest in writing this note primarily stems from the following grounds. First, we provide an alternate self-contained pathway to the proof of $n$-point spin correlation convergence on full-plane in massive scaling limit and isomonodromic deformation: any reader could theoretically refer to this work combined with \cite{cim21} and obtain a complete picture of the justification of those results, which now also hold for isoradial lattices with uniformly non-degenerate angles. In our $\mathbb{R}$-linear formalism, spinors are simply generalizations of ordinary holomorphic functions, which allows us to leverage the theory of generalized analytic functions \cite{bers, vek} for a number of purposes, e.g. to justify formal power series expansions and regularity estimates. Many of these arguments may be straightforwardly applied to bounded domain objects to study their deformations, although (possibly non-explicit) boundary correction terms would be added. Second, we highlight that there are technical improvements to reap from the discrete innovations. Scaling limit process in \cite[IV, Section 4.5]{sato-miwa-jimbo-ii},  \cite[Section 4]{patr81} assumes \emph{exceptional sets}, prohibiting coincidence of certain coordinates. In particular, rotational invariance for $n>2$ case \cite[Theorem 6.6]{patr} contains a similar exception. The scaling functions being the universal limit on generic isoradial lattices, this restriction may now be naturally removed. We also feel apt to mention here the important recent breakthrough on rotational invariance of Fortuin-Kasteleyn percolation models ($1\leq q \leq 4$) based on an entirely different strategy \cite{rot-fk}.

In the remainder of this section, we will state the concrete results we obtain and recall how the discrete model relates to the continuum spinors $f_j$. The next section presents the core continuum argument and, might best be read before the first as an independent piece of complex analysis regarding hypothetical spinors $f_j$ (whose existence is proved in Proposition \ref{prop:phys}). The Appendix contains derivations of necessary estimates using generalized analytic (Bers-Vekua) function theory.

\subsection{Summary of Results}

We first need to recall the definition of the model on isoradial lattice, for which we will generally align with and refer to \cite[Section 2.1]{cim21}.

A planar lattice $\Lambda$ is \emph{isoradial} if every face may be circumscribed by a circle of identical radius $\delta>0$. Its dual $\Lambda^*$ formed by vertices placed on the centers of circumscribing circles is also $\delta$-isoradial, and there is a rhombus $v_0u_0v_1u_1$ associated to every edge $e$ formed by two pairs of incident primal ($v_{0,1}$) and dual vertices ($u_{0,1}$): denote the half-angle at the primal vertices by $\bar\theta_e\in(0,\pi/2)$ (see \cite[Figure 2]{cim21}), which is assumed to be \emph{uniformly bounded} away from the two ends of the interval $0,\pi/2$. In other words, data given by $\Lambda$ is equivalent to a \emph{rhombus tiling} of $\mathbb C$.

Given a finite isoradial graph $G \subset \Lambda$, the \emph{Ising model with Z-invariant weights} \cite{bax78, bax86} on $G$ with \emph{nome} $q \in \mathbb{R}$ (recall the corresponding \emph{elliptic modulus} $k$, \emph{quarter period} $K(k)$, and Jacobi's elliptic function $\operatorname{sn}(\cdot | k)$ \cite[Section 22.2]{dlmf}) is defined as the probability measure $\mathbb{P}_{G}$ on assignments $\sigma$ of $\pm 1$ spins on the \textbf{faces} of $G$ which is given by
\begin{equation*}
\mathbb{P}_{G} [\sigma] \propto \prod_{e \in L(\sigma)} \tan\frac{\hat\theta_e}{2};\; \sin \hat\theta_e := \operatorname{sn}\left(\frac{2K(k)}{\pi} \bar\theta_e |k \right),
\end{equation*}
where the product is over the \emph{low-temperature expansion} representation $L(\sigma)$, i.e. the set of edges across which spins differ (equivalently, $\sigma_{u_0}\sigma_{u_1} = -1$). To be precise, this is the definition of the model with \emph{free boundary condition}, and we may impose $+$ boundary condition by prescribing $\sigma_u = 1$ for faces $u$ in $\Lambda \setminus G$ and taking into account interactions on the boundary of $G$. We will write $\mathbb{P}^{\mathtt{f}}_{G},\mathbb{P}^{+}_{G}$ for disambiguation.

In \emph{massive scaling limit}, we consider the model (notated $\mathbb{P}_{\Lambda}^{(m),+}$, etc.) on progressively finer lattices $\Lambda$ with mesh sizes $\delta \downarrow 0$ with rescaled nome $q=m\delta/2$ for a fixed \emph{mass} $m\in \mathbb R$. Concretely speaking, we assume that $\Lambda=\Lambda_\delta$ is chosen each $\delta$ and consider spins $\sigma_a$ for points $a \in \mathbb{C}$ on a face of $\Lambda$ containing or adjacent to $a$. In any case, the convergence is uniform across isoradial lattices having the same uniform angle bound. For the familiar homogeneous ($J\equiv 1$) model on the square lattice $\Lambda_\delta = \delta \mathbb Z^2$, this scaling of $q$ is equivalent \cite[(2.4)]{cim21} (to first order) to scaling the inverse temperature $\beta=\beta_c-m\delta/2$, with $\beta_c = \frac{1}{2} \ln (1+\sqrt{2})$. In the following, we will exclusively describe the model using parameters $q, m$, in order to repurpose the variables $k$ and $\beta$.

We will now present the main theorem statement. We customarily consider subcritical ($q<0$) and supercritical ($q>0$) models respectively on $\Lambda$ and $\Lambda^*$ (that is, on \emph{vertices} of $\Lambda$), and fix $m<0$ from here now on, because of conventions from the scaling limit setup, but recall that the primal and dual lattices are both isoradial and on an equal footing. The first part is an adaptation of \cite[Theorem 1.3]{cim21} for the full-plane via infinite volume limit; note that similar limit has to be taken in discrete just to define the correlation on the infinite lattice. Both essentially depend on an analogous decorrelation argument which we elaborate in the proof and the next subsection. The second part refers to our reproduction of the SMJ system of equations.

\begin{thm}\label{thm:main}
Let $n\geq 2,m<0$. There are full-plane \emph{scaling functions} $\left< \sigma_{a_1}\sigma_{a_2}\cdots \sigma_{a_n}\right>^{(m),+}, \left< \sigma_{a_1}\sigma_{a_2}\cdots \sigma_{a_n}\right>^{(-m),\mathtt{f}}$ to which the Z-invariant Ising spin correlations converge in massive scaling limit: put $\mathcal{C} = 2^{1/6}e^{3\zeta'(-1)/2}$, then 
\begin{align*}
\delta^{-n/8} \mathbb{E}^{(m),+}_{\Lambda} \left[ \sigma_{a_1}\sigma_{a_2}\cdots \sigma_{a_n}\right] &\to \mathcal{C}^n\left< \sigma_{a_1}\sigma_{a_2}\cdots \sigma_{a_n}\right>^{(m),+};\\
\delta^{-n/8} \mathbb{E}^{(-m),\mathtt{f}}_{\Lambda^*} \left[ \sigma_{a_1}\sigma_{a_2}\cdots \sigma_{a_n}\right] &\to \mathcal{C}^n \left< \sigma_{a_1}\sigma_{a_2}\cdots \sigma_{a_n}\right>^{(-m),\mathtt{f}}.
\end{align*}
These scaling functions are smooth and their logarithmic derivatives are solutions of a closed system of differential equations in $a_1,\ldots,a_n \in \mathbb C$. 
\end{thm}
\begin{proof}
\cite[Theorem 1.3]{cim21} proves the analog 
\begin{equation*}
\delta^{-n/8} \mathbb{E}^{(m),+}_{\Omega^\delta} \left[ \sigma_{a_1}\sigma_{a_2}\cdots \sigma_{a_n}\right] \to \mathcal{C}^n\left< \sigma_{a_1}\sigma_{a_2}\cdots \sigma_{a_n}\right>^{(m),+}_\Omega,
\end{equation*}
where $\Omega$ is a smooth, bounded simply connected domain in $\mathbb{C}$. Given \eqref{eq:pf} and Proposition \ref{prop:phys}, it is easy to see that similar convergence holds for the supercritical correlation with free boundary condition. Then \eqref{eq:rsw} applied to both sides, with $\epsilon$ taken arbitrarily small, gives well-definedness of the infinite volume limit of the correlation, the full-plane scaling function defined by the monotonically decreasing limit of the right hand side as $\Omega \uparrow \mathbb{C}$, and then finally convergence of full-plane correlations.

In Section \ref{sec:2} we derive a closed system of differential equations involving smooth functions $\mathcal{A}_{jk}, \mathcal{B}_{jk}$ in positions $a_1,\ldots,a_n$: see in particular \eqref{eq:Ah}, \eqref{eq:Bh}. These fully characterize the logarithmic derivatives of both scaling functions given Proposition \ref{prop:phys}.
\end{proof}

From the main theorem we derive the following two important consequences.

\begin{cor}
The scaling functions $\mathcal{C}^n\left< \sigma_{a_1}\sigma_{a_2}\cdots \sigma_{a_n}\right>^{(m),+}, \mathcal{C}^n\left< \sigma_{a_1}\sigma_{a_2}\cdots \sigma_{a_n}\right>^{(-m),\mathtt{f}}$ are rotationally invariant.
\end{cor}
\begin{proof}
As mentioned in the introduction, this is a natural consequence of the fact the convergence in Theorem \ref{thm:main} holds for general collections of increasingly finer isoradial lattices. Alternatively, a `continuum' proof (which works, e.g. only assuming convergence on square lattice) in the vein of \cite[Theorem 6.6]{patr} and \cite[Theorem 6.5.5]{palmer} may be given by \eqref{eq:logc} and noticing that $\mathcal{A}_{k,k}{a_k}$ is invariant under rotation (using $\partial_\phi \mathcal{A} = -i\mathcal{A}$, see Section \ref{subsec:deform}).
\end{proof}

\begin{cor}[\cite{wmtb, kako80, patr}]
The scaling functions $\mathcal{C}^2\left< \sigma_{a_1}\sigma_{a_2}\right>^{(m),+}, \mathcal{C}^2\left< \sigma_{a_1}\sigma_{a_2}\right>^{(-m),\mathtt{f}}$ for $n=2$ may be expressed in terms of a Painlev\'e III transcendent,
\begin{align*}
\mathcal{C}^2\left< \sigma_{a_1}\sigma_{a_2}\right>^{(m),+} &= (8|m|)^{1/4}\cosh h_0(|m|r) \exp \int_{\infty}^{|m|r} \frac{ (h'_0(r))^2}{4}-4 \sinh^2h_0(r) dr\\
\mathcal{C}^2\left< \sigma_{a_1}\sigma_{a_2}\right>^{(-m),\mathtt{f}} &=(8|m|)^{1/4}\sinh h_0(|m|r) \exp \int_{\infty}^{|m|r} \frac{ (h'_0(r))^2}{4}-4 \sinh^2h_0(r) dr dr,
\end{align*}
where $\eta_0=\exp(-2h_0)$ satisfies the Painlev\'e III equation $\eta_0'' = \eta_0^{-1}(\eta'_0)^2 - r^{-1}\eta'_0+ \eta_0^3 - \eta_0^{-1}$ \cite[(2.36)]{wmtb} and has the asymptotic $\eta_0(\theta) \sim 1- \frac{ 2}{\pi}K_0(2\theta)$ as $\theta \to \infty$, where $K_\nu$ is the modified Beseel function of the second kind. This asymptotic fixes $\eta_0$ within the one-parameter family of solutions constructed in \cite{mctw}.
\end{cor}
\begin{proof}
We may assume $a_1 = 0, a_2 = r$. From \eqref{eq:logc} and \eqref{eq:betafinal}:
\begin{equation*}
\mathcal{C}^2\left< \sigma_{a_1}\sigma_{a_2}\right>^{(m),+} = (8|m|)^{1/4} \exp \int_{\infty}^{r} \frac{ (h'(tr))^2}{4}-4m^2 \sinh^2h(tr) +\tanh(tr) h'(tr) d(tr),
\end{equation*}
 so integrating $\tanh$ term explicitly and setting $h_0(|m|r) = h(r)$, $\eta_0(|m|r) = \eta(r)$ gives the first result. Then Proposition \ref{prop:phys} and $\beta = \tanh h$ gives the second. See \eqref{eq:asympto} for the estimate at infinity.
\end{proof}

\subsection{Relating the Discrete Model}\label{subsec:12}

\subsubsection*{Note on Infinite Volume Limit}

As mentioned above, to speak of correlations on the infinite $\Lambda$, we need to take infinite volume limit. Recall that standard monotonicity arguments mean that taking the limit along any chosen sequence of growing domains (with either the plus or free boundary condition fixed); what is less obvious is that this limit is unique. As in \cite[Section 2.3]{cim21}, this can be done by appealing to the following Russo-Seymour-Welsh (RSW) type estimate: given $\epsilon>0$, there exists a large enough $A=A(\epsilon)$ such that subcritical Ising model on a simply connected subdomain of $\Lambda$ with plus boundary condition containing an annulus of aspect ratio $A$ has a loop of plus spins within the annulus (separating the inner and outer boundaries) with probability at least $1-\epsilon$. This formulation is a simple consequence of the RSW type estimate for the \emph{critical} FK-Ising model on isoradial lattice (originally due to \cite{chsm2012}; see \cite[Section 2.3]{cim21}) via coupling to the spin model and monotonicity.

In other words, the uniform estimate \cite[(2.10)]{cim21} holds for n-point correlations: if faces $a_1, \ldots, a_n$ are in the discretized ball $\Lambda_D$ of radius $D$ on $\Lambda$, then for the \textbf{subcritical} model,
\begin{equation}\label{eq:rsw}
(1-\epsilon)  \mathbb{E}^+_{\Lambda_{AD}} [\sigma_{a_1}\cdots\sigma_{a_n}] \leq \mathbb{E}^+_{G} [\sigma_{a_1}\cdots\sigma_{a_n}] \leq \mathbb{E}^+_{\Lambda_{AD}} [\sigma_{a_1}\cdots\sigma_{a_n}],
\end{equation}
where $G\subset \Lambda$ is any isoradial domain containing $\Lambda_{AD}$, including any infinite volume limit on $\Lambda$. As \cite[Remark 2.6]{cim21} notes, similar estimate holds for the supercritical model with free boundary condition.

\subsubsection*{Massive Holomorphic Spinors}

We study the correlation functions in the Ising model under scaling limit using \emph{massive holomorphic spinors}. By \emph{massive holomorphicity}, we mean the notion of perturbed holomorphicity $\overline{\partial} f = m\bar{f}$ for a constant $m<0$, where $\overline{\partial} := (\partial_x + i\partial_y)/2$ is one of the two Wirtinger derivatives on $\mathbb{C} \cong \mathbb{R}^2$.

By \emph{spinor}, we mean functions which have $-1$ multiplicative monodromy around given $n\geq 2$ points $a_{1},\ldots,a_{n}\in\mathbb{C}$; in other words, they are well-defined on a fixed double cover $\Cdub$ of $\mathbb{C} \setminus \{a_1, a_2, \ldots, a_n\}$ and gains a factor of $-1$ when switching sheets. 

Moreover, by defining the $n$ Ising spinors below, we will henceforth assume that each $\sqrt{z-a_k}$ is locally well-defined near $a_k$ (and varies smoothly under infinitesimal movements of $a_k$); indeed, flipping the sign of $\sqrt{z-a_k}$ is equivalent to replacing $f_k \to -f_k$.

\begin{defn}\label{defn:spinor}
The \emph{massive Ising spinors} are massive holomorphic functions $f_{1},f_{2},\ldots,f_{n}$ defined on the double cover $\Cdub$ uniquely by the following conditions: for $1\leq k \leq n$,
\begin{enumerate}
\item $L^2$ norm $\iint_\mathbb{C} |f_{k}|^2$ is finite (in fact, less than $\pi/2|m|$ by \eqref{eq:l2});
\item $f_{k}(z) \sim \frac{1}{\sqrt{z-a_j}}=:\frac{i\mathcal{B}_{j,j}}{\sqrt{z-a_j}}$ as $z\to a_j$;
\item $f_{k}(z) \sim  \frac{i\mathcal{B}_{j,k}}{\sqrt{z-a_j}}$ for some $\mathcal{B}_{j,k} \in \mathbb{R}$ as $z\to a_j \neq a_k$ (in fact, $\left| \mathcal{B}_{j,k} \right|\leq 1$ by \eqref{eq:l2}).
\end{enumerate}
\end{defn}
The uniqueness of $f_j$ is easy to see from Lemma \ref{lem:unique}. We show the existence in Proposition \ref{prop:phys}, and summarize the minor changes in notation from \cite{par19, cim21} in its proof.

To extract physical quantities out of these spinors, we utilize massive holomorphic \emph{formal powers} $Z^1_\nu, Z^i_\nu$: these are massive holomorphic functions which asymptotically coincide with holomorphic powers $z^{\nu}, iz^{\nu}$  as $z\to 0$ (since massive holomorphicity is $\mathbb R$-linear, we need two independent functions for each $\nu$). We give explicit definitions in \eqref{eq:formalpowers}. Analogously to the holomorphic case, massive holomorphic functions may be expanded in terms of formal powers. The precise result we need is the following analog of Laurent-Taylor expansion:

\begin{prop}\label{prop:laurent}
Given a massive holomorphic spinor $f$ branching around $a$ satisfying $f(z) = O\left( \left|z-a\right|^{\nu} \right)$ for some half-integer $\nu\in \mathbb Z + \frac{1}{2}$ as $z\to a$, there exist unique real numbers $A^1_{\nu}=A^1_{\nu}(a,f), A^i_{\nu}=A^i_{\nu}(a,f)$, and corresponding $A_\nu = A_\nu(a,f) := A^1_\nu + iA^i_\nu$ such that:
	\begin{equation}\label{eq:laurent}
		f(z) - A_\nu (z-a)^\nu = o(|z-a|^\nu)\mbox{ and }f(z) - A^1_{\nu} Z_{\nu}^1 (z-a) - A^i_{\nu} Z_{\nu}^i (z-a) = O\left(|z-a|^{\nu+1}\right)\text{ as }z\to a.
	\end{equation}
We define higher coefficients $A_{\nu+1}=A^1_{\nu+1}+iA^i_{\nu+1},\ldots$ by using the same procedure on the left hand side, etc.
\end{prop}

Thanks to Proposition \ref{prop:laurent}, we may define the following expansion coefficients. First, define the shorthand $\bullet$ for the real combination $\mathcal{A}_{j,k} \bullet Z_{1/2} := \mathcal{A}^1_{j,k} Z^1_{1/2} + \mathcal{A}^i_{j,k} Z^i_{1/2}$, etc.

\begin{defn}\label{defn:AC}
For each $j,k\in \{1,2,\ldots,n\}$, define $\mathcal{A}_{j,k} := \mathcal{A}^{1}_{j,k} + i\mathcal{A}^{i}_{j,k}, \mathcal{D}_{j,k} := \mathcal{D}^{1}_{j,k} + i\mathcal{D}^{i}_{j,k}$ from the following expansion:
\begin{align}\label{eq:AC}
f_k(z) = \left( i\mathcal{B}_{j,k} \right) \bullet Z_{-1/2}(z-a_j) + \mathcal{A}_{j,k} \bullet Z_{1/2}(z-a_j) +\mathcal{D}_{j,k}\bullet Z_{3/2}(z-a_j) +O\left(|z-a_j|^{5/2} \right)\text{ as }z\to a_j. 
\end{align}
In particular, the coefficients are functions of $a_1,\ldots,a_n$, and we write $\mathcal{A}_{j,k}= \mathcal{A}_{j,k}(a_1,\ldots, a_n)$, etc.
\end{defn}

\subsubsection*{Probabilistic Interpretation of the Coefficients}

As stated above, the coefficients $\mathcal{A}_{j,k}, \mathcal{B}_{j,k}$ encode information about the scaling limit of Ising discrete correlations, and showing that they satisfy a closed system of deformation equations (that is, there is an eventual cancellation of the higher coefficients $\mathcal{D}_{j,k}$) in the positions $a_1,\ldots,a_n$ is a main goal of this paper. This correspondence between the Ising correlations and continuous spinors is the core scaling limit result we utilize from \cite{cim21}. Concretely, we use the following correspondence.

\begin{prop}\label{prop:phys}

The massive Ising spinors from Definition \ref{defn:spinor} exist for any $m<0$. In addition, their coefficients from Definition \ref{defn:AC} satisfy
\begin{align*}
\mathcal{A}_{k,k} =4\partial_{a_k}\log \left< \sigma_{a_1}\sigma_{a_2}\cdots \sigma_{a_n}\right>^{(m),+} ;\quad
|\operatorname{Pf} [\ree \mathcal{B}]| = \frac{\left< \sigma_{a_1}\sigma_{a_2}\cdots \sigma_{a_n}\right>^{(-m),\mathtt{f}}}{\left< \sigma_{a_1}\sigma_{a_2}\cdots \sigma_{a_n}\right>^{(m),+}},
\end{align*}
where $[\mathcal{B}]=[\mathcal{B}_{j,k}]_{j,k}$ is the matrix of coefficients $\mathcal{B}_{j,k}$ and $\operatorname{Pf}$ refers to the Pfaffian of the antisymmetric (by \eqref{eq:hermitian}) matrix $ [\ree \mathcal{B}]$.
\end{prop}
\begin{proof}
First, let us note carefully small differences in convention: our spinors $f_j$ are normalized identically to \cite{chelkak-hongler,par19}, while it is $e^{i\pi/4}$ times the s-holomorphic spinors studied in \cite{cim21}. Coefficients $\mathcal{A}_{j,k}$ are twice the corresponding quantities in all these works (see \cite[(2.22)]{chelkak-hongler} and \cite[(4.20)]{cim21}). Similarly, while $\mathcal{B}_{j,k}$ for $j\neq k$ is always taken to be real, we additionally write $\mathcal{B}_{j,j}:=-i$ for notational convenience.

In smooth bounded simply connected domains $\Omega$, \cite[Theorem 4.17]{cim21} (see also the proof of \cite[Theorem 1.3]{cim21} for applicability for the $n>2$ case) shows that discrete spinors converge to bounded domain analogs $f_j^\Omega$ which continuously takes the boundary condition $f_j^\Omega \sqrt{n_{\mathtt{out}}} \in \mathbb R$ on (the double cover of) $\partial \Omega$, where $n_{\mathtt{out}}$ is the unit outer normal. It also satisfies all conditions from Definition \ref{defn:spinor} (applying \eqref{eq:gr} as in \eqref{eq:l2}).

Since $f_j^\Omega$ are uniformly $L^2$-bounded, taking infinite volume limit along domains $\Omega \uparrow \mathbb{C}$, by \eqref{eq:mvt} we may apply Arzel\`a-Ascoli theorem on compact subsets of $\Cdub$ to extract a subsequence which converges uniformly. Choosing an increasing sequence of compact subsets and diagonalizing, we may extract a subsequence along which $f_j^\Omega$ locally uniformly converges to a limit on all of $\Cdub$: any limit has to satisfy all three conditions of Definition \ref{defn:spinor}, so is unique.

If $\mathcal{A}^\Omega_{j,k}$ is defined as in Definition \ref{defn:AC} from $f^\Omega_j$, \cite[(1.6)]{cim21} gives that
\begin{equation}\label{eq:logbd}
\frac{\left< \sigma_{w}\sigma_{a_2}\cdots \sigma_{a_n}\right>^{(m),+}_\Omega}{\left< \sigma_{z}\sigma_{a_2}\cdots \sigma_{a_n}\right>^{(m),+}_\Omega} = \exp \int_z^w \ree \left[ \frac{1}{2} \mathcal{A}^\Omega_{1,1}(a_1, \ldots, a_n) da_1 \right],
\end{equation}
where $\int_z^w $ is the line integral along any contour from $z$ to $w$ in $\Omega$. As $\Omega \uparrow \mathbb{C}$, LHS converges to the ratio of full-plane scaling functions; $\mathcal{A}^\Omega_{1,1}(a_1, \ldots, a_n)$ for $a_1$ on compact sets away from $a_2, \ldots, a_n$ uniformly converges to $\mathcal{A}_{1,1}(a_1, \ldots, a_n)$ by \eqref{eq:res} since $f^\Omega_1$ converges $f_1$ locally uniformly. That is, (without loss of generality) $\mathcal{A}_{1,1}$ is the logarithmic derivative as desired.

Now we move on to the second Pfaffian identity (we may assume $n$ is even since both sides naturally vanish if $n$ is odd). This identity is an expression of the standard fact that \emph{free fermion} correlations decompose into a Pfaffian of 2-point correlations and dual spin correlations may be obtained as fermion correlations in the Kadanoff-Ceva sense \cite{kadanoff-ceva} weighted by primal spins (see also, e.g. \cite{chi21}). For an informed reader, it should therefore be clear that \emph{some} identity of this type should hold for suitable sign choices of $\pm \mathcal{B}_{j,k}$: indeed, $n$-spin weighted 2-point fermion correlations are generally sign-ambiguous, and converge to similarly sign-ambiguous $\mathcal{B}_{j,k}$ (\cite[(4.34)]{cim21} carries out $n=2$ case, but convergence \emph{up to sign} in the $n>2$ case is fully analogous). 

Therefore, we find the central task here is to check that there is a coherent way to define (i.e. fix signs of) these 2-point correlations such that they appear in the Pfaffian decomposition and converge to $\mathcal{B}_{j,k}$, whose signs we have already chosen. We present an argument based on the combinatorial correspondence of \cite{cck}, which requires elementary but unpleasantly diligent conversions among various combinatorial representations and conventions; for readers more familiar with analysis of massive discrete s-holomorphic spinors, we note that a discrete complex analytic approach as in the proof of (critical case) \cite[Proposition 2.24]{chi21} is still available.

 For an isoradial discretization $\Omega^\delta$ of large $\Omega$, \cite[(3.10), Proposition 3.3]{cck} gives (valid for any mass $m$)
\begin{equation}\label{eq:pf}
\left|\left<\chi_{c_1}\cdots \chi_{c_n} \right>_{[a_1,\ldots,a_n]}\right|=\left|\operatorname{Pf}\left[ \left<\chi_{c_j}\chi_{c_k}\right>_{[a_1,\ldots,a_n]} \right]_{j,k} \right|= \frac{\mathbb{E}^+_{\Omega^\delta}\left[ \mu_{b_1}\mu_{b_2}\cdots \mu_{b_n}\right]}{\mathbb{E}^+_{\Omega^\delta}\left[ \sigma_{a_1}\sigma_{a_2}\cdots \sigma_{a_n}\right]}.
\end{equation}
$b_j$ is a primal vertex incident to $a_j$, and $c_j = (a_jb_j)$ is the adjacent primal-dual vertex pair, or a \emph{corner}. $\mu_b$ is a \emph{disorder insertion} at a primal vertex $b$, and the pure disorder correlation is equal to spin correlation of the dual model \cite[(2.5)]{cim21} under Kramers-Wannier duality \cite{kramers-wannier}, so RHS is nonnegative and converges to the desired ratio of the scaling functions (first in $\Omega$, then taken to infinite volume limit). The \emph{Grassmannian} correlations on the left require \cite[(3.5), (3.10)]{cck} fixing $n$ lifts of $c_j$ on $\Cdub$ and $n$ square roots $\eta_{c_j}$ (also defined in \cite[(2.1)]{cim21}, remembering to adjust $\varsigma = i$ to match \cite{chelkak-hongler, par19}).

Thanks to \cite[Proposition 3.3]{cck}, $\left<\chi_{c_j}\chi_{c}\right>_{[a_1,\ldots,a_n]}$ is equal to the real spinor $\frac{\mathbb{E}^+_{\Omega^\delta}\left[ \chi_c \sigma_{a_1}\cdots \sigma_{a_{j-1}}\mu_{b_j}\sigma_{a_{j+1}}\cdots \sigma_{a_n} \right]}{\mathbb{E}^+_{\Omega^\delta}\left[ \sigma_{a_1}\sigma_{a_2}\cdots \sigma_{a_n}\right]}$, i.e. the $n$-point version of \cite[(4.32)]{cim21} (cf. also the branching structure of the latter, \cite[Remark 2.2]{cim21}). From the real spinor we may define discrete s-holomorphic spinors $F^{\Omega^\delta}_j$ on edges (through projection relations \cite[Definition 3.1]{cim21}) which converge (up to sheet choice) to $ \left(\frac{2}{\pi}\right)^{1/2}f^\Omega_j$ \cite[Theorem 4.17]{cim21}. In fact, we may consider $\eta_c \left<\chi_{c_j}\chi_{c}\right>_{[a_1,\ldots,a_n]}$ as a function of general corners $c=(ab)$ of $G$ lifted up to $\Cdub$, since $\bar{\eta}_c$ in $\tau(Q)$ is canceled in \cite[(3.5), (3.10)]{cck}. From the same definition it is clear that defining $n$ functions $\eta_c \left<\chi_{c_j}\chi_{c}\right>_{[a_1,\ldots,a_n]}$ and $F^{\Omega^\delta}_j$ for $j=1,\ldots,n$ on the same double cover $\Cdub$ is possible. Then $n$ functions $F^{\Omega^\delta}_j$ may also be defined on the same double cover, since the projection relation \cite[(3.1)]{cim21} may be lifted to the double cover with $\hat{\eta}_{c,z}/\eta_c$ always having positive real part. 

Let us be specific. Recall we are considering choices of graph faces $a_j$ which converge to fixed continuum points $a_j \in \mathbb C$. Then we choose incident vertices $b_j$, and lifts of $c_j, b_j$ on $\Cdub$ (branching around discrete faces $a_j$). Then we may fix the sign $\eta_{c_j}=\underline{\eta}_{c_j}$ and define $\eta_c \left<\chi_{c_j}\chi_{c}\right>_{[a_1,\ldots,a_n]}$ and $F^{\Omega^\delta}_j$ so that we have convergence to $ \left(\frac{2}{\pi}\right)^{1/2} f^\Omega_j$ for each $j$. We would like to extract $\mathcal{B}^\Omega_{j,k}$ from the residue at $a_k$ of $f^\Omega_j f^\Omega_k$, as in the proof of \cite[Theorem 4.17]{cim21}.

First, we need to carefully check that $F^{\Omega^\delta}_{k},F^{\Omega^\delta}_{j}$ respectively satisfies the condition for $F^{(m),\delta}_1, F^{(m),\delta}_2$ in \cite[Lemma 3.14]{cim21}. Following the combinatorial definition, both are spinors on the same $\Cdub$ but the former has the $\pm$-type `s-holomorphic singularity' at $c_k$ (see \cite[(2.2)]{par19}, \cite[Lemma 3.2]{chelkak-hongler}), effectively moving the monodromy to $b_k$. A convenient way of defining the spinors $X_1, X_2$ on the `slightly different' $\Upsilon^\times_{[b_k]}, \Upsilon^\times_{[a_k]}$ is as follows: take the already defined spinors $\eta_c \left<\chi_{c_k}\chi_{c}\right>_{[a_1,\ldots,a_n]}, \eta_c \left<\chi_{c_j}\chi_{c}\right>_{[a_1,\ldots,a_n]}$ on common $\Cdub$, and multiply by common $\bar{\eta}_c$ on $\Upsilon^\times$. At the singularity $c_k$ for $\eta_c \left<\chi_{c_k}\chi_{c}\right>_{[a_1,\ldots,a_n]}$, use the `right-side' value $\underline{\eta}_{c_k}$. This creates real-valued spinors on identical $ \Upsilon^\times_{[a_k]}$, but the propagation equation for $X_1$ fails around $z^{-}$: re-wiring as in \cite[Figure 4]{cim21} to $\Upsilon^\times_{[b_k]}$ fixes this.

All in all, LHS in \cite[(3.17)]{cim21} is exactly $2\bar{\eta}^2_{c_k} \left(\underline{\eta}_{c_k} \right) \left(\eta_{c_k} \left<\chi_{c_j}\chi_{c_k}\right>_{[a_1,\ldots,a_n]}\right)$, which is equal to $2\left<\chi_{c_j}\chi_{c_k}\right>_{[a_1,\ldots,a_n]}$ defined using pre-determined $\underline{\eta}_{c_k}, \underline{\eta}_{c_j}$. The RHS $-\frac{1}{2}\oint^{[(m),\delta]} \ree[F^{\Omega^\delta}_{k}F^{\Omega^\delta}_{j}] dz$ (adjusting by $e^{i\pi/4}$ for each factor converts $\imm$ to $-\ree$) converges \cite[Remark 3.4]{cim21} to $-\frac{1}{\pi}\oint_{a_k} \ree[ i\mathcal{B}^\Omega_{j,k}/(z-a_k) dz]=2\mathcal{B}^\Omega_{j,k}$, so we have the desired Pfaffian identity in bounded domains.

As above $\mathcal{B}^\Omega \to \mathcal{B}$ as $\Omega\uparrow \mathbb C$, so we have the desired identity.
\end{proof}

\begin{rem}
The previous preposition relied on infinite-volume limit of spin correlations, scaling functions, and continuous spinors, but \emph{did not} mention discrete spinors defined on the whole $\Lambda$: this is a deliberate choice made in order to limit the length of presentation. Nonetheless, it is both natural and accurate to think about $f_j$'s as limits of analogous discrete full-plane spinors, defined by infinite volume limit; see \cite[Section 2.2]{par19} for the square lattice case.
\end{rem}

Given Proposition \ref{prop:phys}, one can write a formula in the full-plane like \eqref{eq:logbd}, noting that thanks to an RSW type argument as in \eqref{eq:rsw}, a spin `at infinity' ($m<0$) factors to an independent weight of $(-8m\delta)^{1/8}$ by \cite[Proposition 1.5]{cim21} and, e.g. \cite[(19.5.5)]{dlmf}. For example, using radial integral for real $t$:
\begin{equation}\label{eq:logc}
\mathcal{C}^n\left< \sigma_{a_1}\sigma_{a_2}\cdots \sigma_{a_n}\right>^{(m),+} = (-8m)^{n/8} \exp \int_{\infty}^{1} \ree \left[ \frac{1}{2} \sum^{n}_{k=1}\mathcal{A}_{k,k}(ta_1, \ldots, ta_n)a_k dt \right].
\end{equation}

\section{Deriving the Deformation Equation}\label{sec:2}
We now carry out the main analysis. The core idea is that the derivatives of the spinors with respect to the spin positions $a_1, \ldots, a_n$ (preserving the $-1$ monodromy---hence \emph{isomonodromic deformation}) are themselves massive holomorphic functions with at most $3/2$-order poles at each monodromy, and therefore may be expressed in terms of the spinors and their (static) derivatives.

\subsection{Necessary Groundwork}

We first introduce the necessary ingredients and verify the assumptions in the proof strategy, namely the differentiability under deformation and linear independence of the static derivatives.

\subsubsection*{Explicit Formal Powers}

We first give an explicit definition for massive holomorphic formal powers. For half-integers $\nu$ (while we assumed $m<0$, the formula below works for any $m\in \mathbb R$)
\begin{equation}\label{eq:formalpowers}
Z^1_\nu (z) = \frac{\Gamma(\nu+1)}{|m|^\nu} \left( W_\nu(z) +(\sgn m) \overline{W_{\nu+1}(z)} \right);\quad Z^i_\nu =i \frac{\Gamma(\nu+1)}{|m|^\nu} \left( W_\nu(z) -(\sgn m)  \overline{W_{\nu+1}(z)} \right),
\end{equation}
where $W_\nu(z=re^{i\theta}) := e^{i\nu\theta}I_\nu (2|m|r)$ and $I_\nu$ is the modified Bessel function of the first kind. Their asymptotic $Z^1_\nu(z)\sim z^\nu, Z^i_\nu(z)\sim iz^\nu$ as $z\to 0$ is simple to verify from \cite[§10.30]{dlmf}.

Massive holomorphicity for $Z^1_{\nu},Z^i_\nu$ is easy to see from $\partial W_\nu = |m| W_{\nu-1}, \overline{\partial} W_\nu = |m| W_{\nu+1}$, which in turn is straightforward to verify from the standard identity $I_\nu'(r) = I_{v\pm 1}(r)\pm\frac{\nu}{r}I_\nu(r)$ \cite[§10.29]{dlmf} and $\partial =\frac{1}{2}(\partial_x - i\partial_y)= \frac{e^{-i\theta}}{2}(\partial_r - ir^{-1} \partial_\theta)$ and $\overline{\partial} =\frac{1}{2}(\partial_x + i\partial_y)= \frac{e^{i\theta}}{2}(\partial_r + ir^{-1} \partial_\theta)$.

In fact, we record here (using $\partial_x = \partial + \overline{\partial}, \partial_y = i(\partial - \overline{\partial})$):
\begin{align}\label{eq:formalder}
\partial_xZ^1_\nu = \nu Z^1_{\nu-1} + \frac{m^2}{\nu+1}Z^1_{\nu+1};&\quad \partial_xZ^i_\nu = \nu Z^i_{\nu-1} + \frac{m^2}{\nu+1}Z^i_{\nu+1};\\ \nonumber
\partial_yZ^1_\nu = \nu Z^i_{\nu-1} - \frac{m^2}{\nu+1}Z^i_{\nu+1};&\quad \partial_yZ^i_\nu = -\nu Z^1_{\nu-1} + \frac{m^2}{\nu+1}Z^1_{\nu+1}.
\end{align}

\subsubsection*{Green-Riemann theorem}

A natural and powerful tool to work with massive holomorphic functions $f, g$ is the classical Green-Riemann's theorem: around smooth domains $C\subset \mathbb C$ we have
\begin{equation}\label{eq:gr}
\oint_{\partial C} fg dz = 2i \iint_C \overline{\partial} [fg]  = 2i \iint_C 2m \ree[f\overline{g}]  = 2i \iint_C 2m \left<f,g\right>  \in i\mathbb R,
\end{equation}
where $\left< f,g \right>$ refers to the inner product of the complex values $f,g$ as vectors in $\mathbb R^2$. The product of two spinors having $-1$ multiplicative monodromy around the same point is a well-defined function in $\mathbb C$, rather than the double cover, and thus we may use the same formula.

\subsubsection*{Cauchy-type Half-Integer Residue Formula}

One crucial consequence of \eqref{eq:gr} is a residue formula using the formal powers akin to the holomorphic case. Specifically, it implies that the real part of the contour integral $\oint_{\gamma} fg dz$ remains invariant on any loop homotopic to $\gamma$ in the domain of massive holomorphicity. So suppose a spinor $f$ has the asymptotic \eqref{eq:laurent} near $a$, then we have the following formula
\begin{equation}\label{eq:res}
A^1_\nu(a,f) =-\frac{1}{2\pi} \ree \oint_\gamma f\cdot Z^i_{-\nu-1} dz;\quad A^i_\nu(a,f) =-\frac{1}{2\pi} \ree \oint_\gamma f \cdot Z^1_{-\nu-1} dz,
\end{equation}
which is easily verified by setting $\gamma = \partial B_r (a)$ and studying the asymptotic of the integrand as $r \downarrow 0$.

Namely, it is simple to verify from the definition \eqref{eq:formalpowers} that the real part of the contour integral as in \eqref{eq:gr} along $\partial B_r(a)$ of any pairing $Z_{\nu_1}\cdot Z_{\nu_2}$ is nonzero only for the case $Z^1_{\nu_1}\cdot Z^i_{\nu_2}$ with $\nu_1 + \nu_2 = -1$, when it is one. This calculation also gives that, as in the holomorphic case, the higher coefficients $A^1_{\nu + 1}, A^i_{\nu + 1}, \ldots$ may also be computed simply by coupling $f$ respectively with $Z^i_{-\nu}, Z^1_{-\nu},\ldots$. This formula in particular justifies that the derivatives of the series coefficients of the Ising spinors with respect to the locations $a_j$ exist as soon as the spinors have corresponding derivatives in the bulk.

\subsubsection*{Spinors and their Derivatives Form a Basis}

We now show that the Ising spinors $f_1, \ldots, f_n$ and their derivatives $\partial_x f_1,  \ldots, \partial_x f_n$, $\partial_y f_1, \ldots, \partial_y f_n$ form a basis of the real vector space $V_{-3/2}$, defined as the space of massive holomorphic spinors $f$ on $\Cdub$ having asymptotics (for given fixed $a_1, \ldots, a_n$)

\begin{enumerate}
\item	$f(z) = O\left(|z-a_j|^{-3/2}\right)$ as $z\to a_j$ for all $1\leq j \leq n$;
\item	$f(z) = O(|z|^{-n/2})$ as $|z|\to\infty$.
\end{enumerate}
Define the subspace $V_{-1/2} \subset V_{-3/2}$ analogously by replacing the exponent $-3/2$ with $-1/2$, or equivalently imposing $A_{-3/2}(a_k,f)=0$ for all $k$.

We once again benefit from \eqref{eq:gr}. Suppose we have $f,g \in V_{-1/2}$, then applying \eqref{eq:gr} in $\mathbb C$ (more precisely, excising small disks around $a_j$ as in \eqref{eq:res} and using exponential decay at infinity; note the orientation of the excised disks are reversed) yields
\begin{align} 
2\pi  \left[ \imm \left(A_{-1/2}(a_1,f)A_{-1/2}(a_1,g)\right)+\cdots + \imm \left(A_{-1/2}(a_n,f)A_{-1/2}(a_n,g)\right) \right] &= 0 \label{eq:innerproduct} \\
-2\pi \left[ \ree\left(A_{-1/2}(a_1,f)A_{-1/2}(a_1,g)\right)+\cdots + \ree \left(A_{-1/2}(a_n,f)A_{-1/2}(a_n,g)\right) \right] &= 2 \iint_\mathbb{C} 2m\ree [f\overline{g}], \label{eq:l2}
\end{align}
where we take respectively real and imaginary parts. Note further that \eqref{eq:innerproduct} is valid even in $V_{-3/2}$ as long as we consider combinations of type $A_{1/2}\cdot A_{-3/2}$ as well.

Applying the above on the Ising spinors $f_1, \ldots, f_n$ already yields an important information. Applying \eqref{eq:innerproduct} on $f_j, f_k$, we see
\begin{equation}\label{eq:hermitian}
\ree \mathcal{B}_{j,k} + \ree \mathcal{B}_{k,j} = 0,\mbox{ or the matrix }[i\mathcal{B}]\mbox{ is Hermitian}.
\end{equation}

\begin{lem}\label{lem:unique}
The numbers $A^1_{-1/2}(a_1,f)\ldots A^1_{-1/2}(a_n,f)$ and $A_{-3/2}(a_1,f),\ldots,A_{-3/2}(a_n,f)$ uniquely determine $f \in V_{-3/2}$. In particular, $\dim_\mathbb{R} V_{-1/2}=n$ and $\dim_\mathbb{R} V_{-3/2}$ is at most $3n$.
\end{lem}

\begin{proof}
It suffices to show any $f \in V_{-1/2} \subset V_{-3/2}$ such that $A_{-1/2}(a_k,f)=iA^i_{-1/2}(a_k,f)$ for all $k$ is identically zero. This is easily shown by applying \eqref{eq:l2} and seeing $2\pi \left( |A^i_{-1/2}(a_1,f)|^2 + \cdots +|A^i_{-1/2}(a_n,f)|^2 \right) = \iint_{\mathbb C} 4m|f|^2 \leq 0$, so $f \equiv 0$. Since $f_1,\ldots,f_n$ already span $V_{-1/2}$, the dimension of $V_{-1/2}$ is exactly $n$.
\end{proof}

Now we calculate asymptotics of $\partial_x f_k$ and $\partial_y f_k$. We may differentiate the conditions 2-3 in Definition \ref{defn:spinor} term-by-term and have $O\left(|z-a|^{3/2} \right)$-rate errors; this is a simple consequence of \eqref{eq:mvt}. Using \eqref{eq:formalder}, as $z\to a_j$,
\begin{align}\label{eq:expder}
\partial_x f_k(z) &= -\frac{i\mathcal{B}_{j,k}}{2}\bullet Z_{-3/2}(z-a_j) + \frac{\mathcal{A}_{j,k}}{2}\bullet Z_{-1/2}(z-a_j) + \left(\frac{3\mathcal{D}_{j,k}}{2} + 2im^2 \mathcal{B}_{j,k} \right)\bullet Z_{1/2}(z-a_j) + O\left(|z-a_j|^{3/2} \right);\\ \nonumber\
\partial_y f_k(z) &= \frac{\mathcal{B}_{j,k}}{2}\bullet Z_{-3/2}(z-a_j) + \frac{i\mathcal{A}_{j,k}}{2} \bullet Z_{-1/2}(z-a_j)+ \left(\frac{3i\mathcal{D}_{j,k}}{2}+ 2m^2 \mathcal{B}_{j,k} \right) \bullet Z_{1/2}(z-a_j) + O\left(|z-a_j|^{3/2} \right).
\end{align}

Let us introduce another shorthand. For complex matrix elements (say) $\mathcal{A}_{j,k}$, define the $2n\times 2n$ \textbf{real} matrix $\lr{\mathcal{A}}$ by replacing the complex entries by the $2\times 2$ real matrix in the usual way:
\begin{equation}\label{eq:lr}
 \mathcal{A}_{j,k}  \to 
 \begin{pmatrix}
 \ree \mathcal{A}_{j,k} & -\imm \mathcal{A}_{j,k} \\
 \imm \mathcal{A}_{j,k} & \ree \mathcal{A}_{j,k}
 \end{pmatrix}.
\end{equation} 
Clearly, this matrix may be treated in many ways like the complex matrix $[\mathcal{A}]$: for example, one is invertible if and only if the other is.

\begin{prop}
The spinors $ \partial_x f_1, \partial_y f_1, \ldots, \partial_y f_n$ are linearly independent over $\mathbb R$. Equivalently, $\lr{i\mathcal{B}}$ is invertible.
\end{prop}

\begin{proof}
It suffices to show the column vectors of $2n$ coefficients $A^1_{-3/2}(a_1),A^i_{-3/2}(a_1),\ldots,A^i_{-3/2}(a_n)$  for the given $2n$ functions form an invertible matrix. Given \eqref{eq:expder}, this matrix is seen to be precisely $-\frac{1}{2}\lr{i\mathcal{B}}$.

Let us work with the complex matrix $[i\mathcal{B}]$. Recall that $\mathcal{B}_{j,k}$ is real if and only if $j \neq k$, and $i\mathcal{B}_{k,k} = 1$ for all $k$. Therefore, we may decompose $[i\mathcal{B}] = I + iB$, where $B:=\imm \left[ i\mathcal{B} \right]$ is the $n\times n$ real matrix made of only the off-diagonal elements of $\mathcal{B}_{j,k}$, which is antisymmetric by \eqref{eq:hermitian}.

By standard linear algebra, any nonzero eigenvalue of $B$ comes in the form $i\lambda$, where $\lambda$ is real and there exist a pair of $n$ dimensional real vectors $v,w$ of unit norm such that $Bv = \lambda w, Bw=-\lambda v$. In other words, by definition of $\mathcal{B}_{j,k}$, $f: = v_1 f_1 + \cdots + v_n f_n$ has poles $A_{-1/2}(a_k,f) = v_k + i\lambda w_k$ for all $k$. Now apply \eqref{eq:l2} on $f$, and we get
\begin{equation*}
-2\pi \left[ (v_1^2-\lambda^2 w_1^2) + \cdots (v_n^2 - \lambda^2 w_n^2) \right] = 2\pi(\lambda^2-1) = 2\iint 2m|f|^2 <0,
\end{equation*}
since $v,w$ have unit norm. That is, $B$ has spectral radius strictly smaller than $1$, and $I+iB$ is invertible.
\end{proof}

Now that we have linear independence, we may produce `pure' basis vectors $g^1_{1},g^i_{1},\ldots,g^1_{n},g^i_{n}$ whose coefficients $A^1_{-1/2},A_{-3/2}$ all vanish except $A^1_{-3/2}(a_k,g^1_k)=A^i_{-3/2}(a_k,g^i_k)=1$ for each $k$. 

Studying \eqref{eq:expder}, we see that the coefficients for the required linear combination of $\partial_x f_k, \partial_y f_k$ would be given precisely by the entries of the inverse $\lr{-i\mathcal B/2}^{-1}$, essentially postcomposing this matrix on the $2n\times 2n$-matrix of the $3/2$-order (real and imaginary) expansion coefficients at $a_j$ of $\partial_x f_k, \partial_y f_k$. The entries of this real matrix may be recovered from the real and imaginary parts of the $n\times n$ complex matrix $\left[-i\mathcal B/2\right]^{-1}$ according to \eqref{eq:lr}. Once we have the desired $3/2$-order behavior, we cancel out $A^1_{-1/2}$ using $f_k$. Explicitly, we claim that the following definition works:
\begin{align*}
g^1_{k} &:= \sum_{k'=1}^n \ree \left[-i{\mathcal B}/{2}\right]^{-1}_{k',k} \partial_x f_{k'} + \imm \left[-i{\mathcal B}/{2}\right]^{-1}_{k',k} \partial_y f_{k'} - \ree \left(\left[{\mathcal A}/{2}\right]\left[-i{\mathcal B}/{2}\right]^{-1}\right)_{k',k} f_{k'}; \\
g^i_{k} &:= \sum_{k'=1}^n -\imm \left[-i{\mathcal B}/{2}\right]^{-1}_{k',k} \partial_x f_{k'} + \ree \left[-i{\mathcal B}/{2}\right]^{-1}_{k',k} \partial_y f_{k'} - \ree \left(i\left[{\mathcal A}/{2}\right]\left[-i{\mathcal B}/{2}\right]^{-1}\right)_{k',k} f_{k'}.
\end{align*}
Indeed, noting $\overline{\left[i\mathcal B\right]} =I - i \imm {\left[i\mathcal B\right]}$, it is straightforward to verify the following asymptotics from \eqref{eq:expder}
\begin{align}\label{eq:gasymp}
g^1_k(z) &= I_{j,k} Z^1_{-3/2}(z-a_j) - \imm (\overline{\left[i\mathcal B\right]} \left[{\mathcal A}\right]\left[i{\mathcal B}\right]^{-1})_{j,k} Z^i_{-1/2}(z-a_j) + \left( -4m^2 \left[ i\mathcal{B}\right] \overline{\left[ i\mathcal{B}\right]}^{-1}{-3\left[ \mathcal{D}\right]\left[ i\mathcal{B}\right]^{-1}} \right. \\ \nonumber
&\left. + \left[{\mathcal A}\right] \ree \left( \left[{\mathcal A}\right]\left[i{\mathcal B}\right]^{-1} \right) \right)_{j,k}\bullet Z_{1/2}(z-a_j) + O\left(|z-a_j|^{3/2} \right) ;\\ \nonumber
g^i_k(z) &= I_{j,k} Z^i_{-3/2}(z-a_j) - \ree (\overline{\left[i\mathcal B\right]} \left[{\mathcal A}\right]\left[i{\mathcal B}\right]^{-1})_{j,k} Z^i_{-1/2}(z-a_j) +  \left(4im^2 \left[ i\mathcal{B}\right] \overline{\left[ i\mathcal{B}\right]}^{-1}{-3i\left[ \mathcal{D}\right]\left[ i\mathcal{B}\right]^{-1}} \right. \\ \nonumber
&\left. - \left[{\mathcal A}\right] \imm \left( \left[{\mathcal A}\right]\left[i{\mathcal B}\right]^{-1} \right) \right)_{j,k}\bullet Z_{1/2}(z-a_j) + O\left(|z-a_j|^{3/2} \right).
\end{align}
Using \eqref{eq:innerproduct} (and the comment below) on $f_j$ and $g^1_k,g^i_k$, we get ${\left[{\mathcal A}\right]}_{k,j} = (\overline{\left[i\mathcal B\right]} \left[{\mathcal A}\right]\left[i{\mathcal B}\right]^{-1})_{j,k} $. Since $\overline{\left[i\mathcal B\right]}$ is the transpose of ${\left[i\mathcal B\right]}$, we conclude that $\left[{\mathcal A}\right]\left[i{\mathcal B}\right]^{-1}$ is symmetric.

\subsubsection*{Spinors are Differentiable under Deformation}

Thanks to \eqref{eq:res}, differentiability of the spinors $f_k$ in the bulk under deformation to implies that of the coefficients $\mathcal{A}_{j,k},\mathcal{B}_{j,k}, \ldots$. The actual proof will proceed by carefully alternating bulk and coefficient differentiability.

Without loss of generality, we will show differentiability under the deformation in the $x$-direction $a_1 \to a_1 + h$. Define $f^h_j$ to be the increment of $f_j$ under this change: our goal is to show $h^{-1} f_j^h$ as $h\to 0$ converges to a limit $\partial_h f_j$ in the bulk. We will consistently use $h$ superscripts to denote corresponding {increments} when it is clear. We will use a familiar strategy in the field: we will extract a subsequential limit which must be unique.

We will in fact work on double covers $\Cdub^r$ of $\mathbb C^r := \mathbb{C}\setminus B_r(a_1)$ for small $r>0$ by restriction of the original $\Cdub$ to function as a common subdomain, so that we can consider derivatives of the spinors with respect to $a_1$.

\begin{lem}\label{lem:f1}
For each fixed small $r>0$, we have $\iint_{\mathbb C^r} |f^h_1|^2 = O(h^2)$ as $h\to 0$.
\end{lem}

\begin{proof}
Applying Green-Riemann's formula as in \eqref{eq:l2} on $\mathbb C^{r/2}$ yields
\begin{equation}\label{eq:l2cont}
2\pi \left( -\oint_{\partial B_{r/2}(a_1) }\left( f^h_1\right)^2 dz +\left| \mathcal{B}^h_{2,1}\right|^2 +\cdots +\left| \mathcal{B}^h_{n,1}\right|^2 \right)=2\iint_{\mathbb C^{r/2}} 2m\left| f^h_1\right|^2,
\end{equation}
so it suffices to show $s_h := h^{-2} \int_{\partial B_{r/2}(a_1) }\left| f^h_1\right|^2$ is bounded. So suppose, possibly along a subsequence as $h\to 0$, $s_h \to \infty$. Still, the renormalized $\left( \sqrt{s_h} h \right)^{-1} f^h_1$ is $L^2$-bounded in $\Cdub^{r/2}$, and so by \eqref{eq:mvt}, we may apply Arzel\`a-Ascoli theorem on compact subsets of $\Cdub^{r/2}$. Choosing an increasing sequence of compact subsets and diagonalizing, we may again extract a subsequence along which $\left( \sqrt{s_h} h \right)^{-1} f^h_1$ locally uniformly converges to a limit $g$ on all of $\Cdub^{r/2}$.

However, the local uniform convergence may be extended to all of $\Cdub$: writing $Z^h(z):=Z^1_{-1/2}(z-(a_1+h)) - Z^1_{-1/2}(z-a_1)$ and applying Green-Riemann's theorem on the annulus $B_r(a_1) \setminus B_{2h} (a_1)$,
\begin{equation*}
2\pi \left( \oint_{\partial B_{r}(a_1) }\left(s_h  h^2 \right)^{-1}\left( f^h_1 - Z^h \right)^2 dz - \oint_{\partial B_{2h}(a_1) }\left(s_h  h^2 \right)^{-1} \left( f^h_1 - Z^h \right)^2 dz\right)=2\iint_{B_r(a_1) \setminus B_{2h} (a_1)} \left(s_h  h^2 \right)^{-1} 2m\left| f^h_1 - Z^h \right|^2.
\end{equation*}
Note that the integral over $\partial B_r(a_1)$ converges to $\oint_{\partial B_r(a_1)} g^2 dz < \infty$ since $Z^h = O(h)$ there. The integral over $\partial B_{2h}(a_1)$ vanishes thanks to \eqref{eq:Aunif}. So the $L^2$ norm of $\left( \sqrt{s_h} h \right)^{-1} (f^h_1 - Z^h) $ is uniformly bounded on the annulus (which exhausts $B_r(a_1)\setminus\{ a_1 \}$), and we may yet again extract a subsequence along which $\left( \sqrt{s_h} h \right)^{-1} (f^h_1 - Z^h) $ locally uniformly converges to $g$ (now extended to all of $\Cdub$; note $\left( \sqrt{s_h} h \right)^{-1} Z^h \to 0$ in the bulk).

The fact that $g$ is in $V_{-1/2}$ may also be verified from the $L^2$ boundedness of $\left( \sqrt{s_h} h \right)^{-1}  f^h_1$ (away from $a_1$) and $\left( \sqrt{s_h} h \right)^{-1} (f^h_1 - Z^h)$ (near $a_1$) using massive Cauchy formula (say, the fact that $g$ is massive holomorphic itself) and \eqref{eq:mvt}. Similarly, one may verify that the coefficients $A^1_{-1/2}(a_j,g)=0$ for all $j>1$ using \eqref{eq:res}. Near $a_1$, one may use the fact that one may deform the contour in \eqref{eq:res}: as $h\to 0$, by \eqref{eq:Aunif},
\begin{align*}
 A^1_{-1/2}(a_1,g) &\leftarrow -\frac{\left(\sqrt{s_h}  h \right)^{-1}}{2\pi}\ree \oint_{\partial B_{r}(a_1) }\left( f^h_1 - Z^h \right)(z)Z^i_{-1/2}(z-a_1) dz\\
 &=-\frac{\left(\sqrt{s_h}  h \right)^{-1}}{2\pi} \ree \oint_{\partial B_{2h}(a_1) }\left( f^h_1 - Z^h \right)(z) Z^i_{-1/2} (z-a_1)dz \to 0.
\end{align*}
Therefore $g\equiv 0$ by Lemma \ref{lem:unique}. This means that $\left(s_h  h^2 \right)^{-1} \int_{\partial B_{r/2}(a_1) }\left| f^h_1\right|^2\to 0$, which contradicts the definition of $s_h$.
\end{proof}
Now that we have an $L^2$ bound in $\mathbb C^r$ for each small $r>0$, we may play the same game as in the proof of the bound and yet again extract a subsequence along which $h^{-1}f^h_1$ converges locally uniformly to some $g$ in $\Cdub$. Note that $h^{-1} Z^h$ tends to $-\partial_x Z^1_{-1/2}(z-a_1) = \frac{1}{2}Z^1_{-3/2}(z-a_1) - 2m^2 Z^1_{1/2}(z-a_1)$. Again by similar arguments, $g(z) + \partial_x Z^1_{-1/2}(z-a_1)$ only has at most $1/2$-order poles at $a_1$ and $g$ has purely imaginary $1/2$-order poles at $a_2,\ldots,a_n$; in fact, by alternatively setting $Z^h(z):=Z^1_{-1/2}(z-(a_1+h)) + \left(\mathcal{A}_{1,1}+\mathcal{A}^h_{1,1}\right)\bullet Z_{1/2}(z-(a_1+h)) - Z^1_{-1/2}(z-a_1) - \mathcal{A}_{1,1}\bullet Z_{1/2}(z-a_1)$, it is easy to see that $A^1_{-1/2}(a_1,g) = -\frac{1}{2} \ree \mathcal{A}_{1,1}$ (at least along the subsequence giving the limit $g$, $\mathcal{A}^h_{1,1} = O(h)$ by \eqref{eq:res}). This fixes $g$ uniquely by Lemma \ref{lem:unique}, and we see that $g =\partial_h f_1$ as desired.

Given differentiability of $f_1$, we may repeat the same arguments starting from Lemma \ref{lem:f1} with other spinors $f_k$. The only difference is that we need to set
\begin{equation*}
Z^h(z):=\left(\mathcal{B}_{1,k} + \mathcal{B}^h_{1,k} \right) Z^i_{-1/2}(z-(a_1+h))  - \mathcal{B}_{1,k} Z^i_{-1/2}(z-a_1),
\end{equation*}
and we need to first demonstrate the limit of $h^{-1}  \mathcal{B}^h_{1,k}$ to carry out the proof. However, this is already done by the existence of $\partial_h f_1$: they are precisely the coefficients $-A^i_{-1/2}(a_k,\partial_h f_1)$ by \eqref{eq:hermitian}. Thus we have shown
\begin{prop}\label{prop:isom}
The Ising spinors $f_1,\ldots f_n$ are differentiable under infinitesimal movements of $a_1,\ldots,a_n$.
\end{prop}
\subsection{Derivation of Deformation Equations}\label{subsec:deform}
\subsubsection*{General $n$-point case}
Throughout this section, we will denote by $v=(v_1,\ldots,v_n)\in \mathbb C^n$ the direction of deformation of spin locations $a_1,\ldots,a_n$, and $\mathbf{h}:= vh$ for $h\in \mathbb R$ the corresponding one-parameter family. Under this deformation $\partial_{\mathbf{h}}$, by Proposition \ref{prop:isom} (indeed, it was a crucial building block of its proof) the derivative of a spinor $f_k$ has asymptotic as $z\to a_j$
\begin{align*}
\partial_{\mathbf{h}} f_k &= \frac{\left(iv_j\mathcal{B}_{j,k} \right)}{2}\bullet Z_{-3/2}(z-a_j) + \left(i \partial_\mathbf{h} \mathcal{B}_{j,k} - \frac{v_j \mathcal{A}_{j,k}}{2}\right) \bullet Z_{-1/2}(z-a_j)\\
 &+ \left(\partial_\mathbf{h} \mathcal{A}_{j,k} -\frac{3v_j}{2}\mathcal{D}_{j,k} - 2i m^2 \overline{v_j} \mathcal{B}_{j,k} \right) \bullet Z_{1/2}(z-a_j) + O(|z-a_j|^{3/2}).
\end{align*}
We now identify $\partial_{\mathbf{h}} f_k$ with the corresponding linear combination of $f_k, g^1_k, g^i_k$. Note that $i v_j \mathcal{B}_{j,k}$ is precisely the component of the matrix product $\left[ \diag v \right]\left[ i\mathcal{B} \right] $, and $i \partial_\mathbf{h} \mathcal{B}_{j,k} \in i\mathbb{R}$. So
\begin{equation*}
\partial_{\mathbf{h}} f_k = \sum_{k'=1}^n \frac{1}{2}\left(\left[ \diag v \right]\left[ i\mathcal{B} \right] \right)_{k',k}\bullet g_{k'} - \frac{1}{2}\ree \left(\left[ \diag v \right]\left[ \mathcal{A} \right]\right)_{k',k} f_{k'}.
\end{equation*}
Now, comparing $A^i_{-1/2}$ between the above and the asymptotics \eqref{eq:AC} and \eqref{eq:gasymp}
\begin{equation*}
\left[ i\partial_\mathbf{h} \mathcal{B}\right] - \frac{i}{2}  \imm \left(\left[ \diag v \right]\left[ \mathcal{A} \right]\right)= \frac{i}{2} \imm \left(-\overline{\left[i\mathcal B\right]} \left[{\mathcal A}\right]\left[i{\mathcal B}\right]^{-1}\left[ \diag v \right]\left[ i\mathcal{B} \right]\right)  - \frac{i}{2}\imm \left[ i\mathcal{B} \right]\ree \left(\left[ \diag v \right]\left[ \mathcal{A} \right]\right) .
\end{equation*}
Again since $I -i \imm \left[ i\mathcal{B} \right] = \overline{\left[ i\mathcal{B} \right]}$, we arrive at the following expression
\begin{equation}\label{eq:Bh}
\left[ i\partial_\mathbf{h} \mathcal{B}\right]= \frac{i}{2} \imm \left( \overline{\left[i\mathcal B\right]}\left[ \diag v \right]\left[ \mathcal{A} \right] - \overline{\left[i\mathcal B\right]} \left[{\mathcal A}\right]\left[i{\mathcal B}\right]^{-1}\left[ \diag v \right]\left[ i\mathcal{B} \right]\right)=  \frac{i}{2} \imm \left( \overline{\left[i\mathcal B\right]} \left[\left[ \diag v \right],  \left[{\mathcal A}\right]\left[i{\mathcal B}\right]^{-1} \right] \left[ i\mathcal{B} \right]\right),
\end{equation}
using the matrix commutator $[,]$.

From comparing $A_{1/2}$, we have
\begin{align*}
\left[ \partial_\mathbf{h} \mathcal{A} \right] -\frac{3}{2}\left[ \diag v \right]\left[ \mathcal{D} \right] - 2 m^2 \overline{\left[ \diag v \right]}\left[i \mathcal{B} \right]&= -2m^2 \left[ i\mathcal{B}\right] \overline{\left[ i\mathcal{B}\right]}^{-1}\overline{\left[ \diag v \right]\left[ i\mathcal{B} \right]}{-\frac32\left[ \mathcal{D}\right]\left[ i\mathcal{B}\right]^{-1}}\left[ \diag v \right]\left[ i\mathcal{B} \right]\\
&+ \frac{\left[{\mathcal A}\right]}{2} \ree \left( \left[{\mathcal A}\right]\left[i{\mathcal B}\right]^{-1} \left[ \diag v \right]\left[ i\mathcal{B} \right]\right) - \frac{\left[ \mathcal{A} \right] }{2} \ree \left(\left[ \diag v \right]\left[ \mathcal{A} \right]\right),
\end{align*}
or
\begin{align*}
\left[ \partial_\mathbf{h} \mathcal{A} \right] &= \frac{3}{2}\left[\left[ \diag v \right], \left[ \mathcal{D} \right]\left[i \mathcal{B} \right]^{-1} \right]\left[i \mathcal{B} \right] + 2 m^2 \overline{\left[ \diag v \right]}\left[i \mathcal{B} \right]\\ 
& -2m^2 \left[ i\mathcal{B}\right] \overline{\left[ i\mathcal{B}\right]}^{-1}\overline{\left[ \diag v \right]\left[ i\mathcal{B} \right]}
+ \frac{\left[{\mathcal A}\right]}{2} \ree \left( \left[ \left[{\mathcal A}\right]\left[i{\mathcal B}\right]^{-1} ,\left[ \diag v \right]\right] \left[ i\mathcal{B} \right]\right) .
\end{align*}

Unlike \eqref{eq:Bh}, this equation has terms involving $\mathcal{D}$ which we would like to eliminate. This can be done using rotational symmetry of the spinors: as can be checked directly, spinors maintain massive holomorphicity under rotation $f^\phi_k(z) := f_k(e^{-i\phi}z)e^{-i\phi/2}$. The vector $v$ corresponding to this rotation is precisely $ia := (ia_1, \ldots, ia_n)$, and we may check directly from \eqref{eq:formalpowers} that $\partial_\phi \mathcal{A} = -i\mathcal{A}$. Plugging in the above, we have
\begin{align*}
\left[ -i \mathcal{A} \right] &= \frac{3}{2}\left[\left[ \diag ia \right], \left[ \mathcal{D} \right]\left[i \mathcal{B} \right]^{-1} \right]\left[i \mathcal{B} \right] + 2 m^2 \overline{\left[ \diag ia \right]}\left[i \mathcal{B} \right] \\
&-2m^2 \left[ i\mathcal{B}\right] \overline{\left[ i\mathcal{B}\right]}^{-1}\overline{\left[ \diag ia \right]\left[ i\mathcal{B} \right]}
+ \frac{\left[{\mathcal A}\right]}{2} \ree \left( \left[ \left[{\mathcal A}\right]\left[i{\mathcal B}\right]^{-1} ,\left[ \diag ia \right]\right] \left[ i\mathcal{B} \right]\right).
\end{align*}
In the same deformation we have $i\partial_\phi \mathcal{B} = 0$, so incorporating \eqref{eq:Bh} we may write (again using $\overline{\left[i\mathcal B\right]} = I - i \imm {\left[i\mathcal B\right]}$)
\begin{align}\label{eq:abdiag}
\left[  \mathcal{A} \right]\left[i{\mathcal B}\right]^{-1} &=- \frac{3}{2}\left[\left[ \diag a \right], \left[ \mathcal{D} \right]\left[i \mathcal{B} \right]^{-1} \right] + 2 m^2 \overline{\left[ \diag a \right]}\\ \nonumber
&-2m^2 \left[ i\mathcal{B}\right] \overline{\left[ i\mathcal{B}\right]}^{-1}\overline{\left[ \diag a \right]\left[ i\mathcal{B} \right]}\left[i{\mathcal B}\right]^{-1}
+ \frac{\left[{\mathcal A}\right]\left[i{\mathcal B}\right]^{-1}}{2}  \left[ \left[ \diag a \right],  \left[{\mathcal A}\right]\left[i{\mathcal B}\right]^{-1}\right] .
\end{align}
Note carefully that this equation in fact specifies $\left[ \left[ \diag v \right], \left[ \mathcal{D} \right]\left[i \mathcal{B} \right]^{-1} \right]$ and therefore already produces a closed system of equations for $\left[ \partial_\mathbf{h} \mathcal{A} \right]$ as desired: indeed, commutating a matrix by $\left[ \diag a \right]$ simply multiplies the $(j,k)$-matrix entry by $(a_j - a_k)$, so it suffices to multiply each off-diagonal entry of $\left[ \left[ \diag a \right], \left[ \mathcal{D} \right]\left[i \mathcal{B} \right]^{-1} \right]$  by $\frac{v_j - v_k}{a_j - a_k}$ to obtain the desired commutator with $ \left[ \diag v \right]$.

This step can be essentially written using the Jacobi identity on the commutators and the fact that diagonal matrices commute. We will pursue the variation of rotationally invariant $\left[ \left[ \diag a \right], \left[ \mathcal{A} \right]\left[i \mathcal{B} \right]^{-1} \right]$  \cite[Theorem 6.5.3]{palmer} (this kills the diagonal elements of $\left[\mathcal{A} \right]\left[i \mathcal{B} \right]^{-1} $, but they can be recovered from \eqref{eq:abdiag}). Similarly to above, it is easy to verify
\begin{align*}
&\partial_\mathbf{h}\left( \left[ \mathcal{A} \right] \left[i \mathcal{B} \right]^{-1} \right) = \left[ \partial_\mathbf{h} \mathcal{A} \right] \left[i \mathcal{B} \right]^{-1} -\left[  \mathcal{A} \right] \left[i\mathcal{B} \right]^{-1} \left[i \partial_\mathbf{h}\mathcal{B} \right]\left[i\mathcal{B} \right]^{-1} =\\
  \frac{3}{2}&\left[\left[ \diag v \right], \left[ \mathcal{D} \right]\left[i \mathcal{B} \right]^{-1} \right] 
+ 2 m^2 \overline{\left[ \diag v \right]} -2m^2 \left[ i\mathcal{B}\right] \overline{\left[ i\mathcal{B}\right]}^{-1}\overline{\left[ \diag v \right]\left[ i\mathcal{B} \right]}{\left[ i\mathcal{B}\right]}^{-1}
-\frac{1}{2} \left[  \mathcal{A} \right] \left[i\mathcal{B} \right]^{-1} \left[\left[ \diag v \right],  \left[{\mathcal A}\right]\left[i{\mathcal B}\right]^{-1} \right],
\end{align*}
and finally
\begin{align}\label{eq:Ah}
&\partial_\mathbf{h}\left[ \left[ \diag a \right], \left[ \mathcal{A} \right]\left[i \mathcal{B} \right]^{-1} \right]= \left[ \left[ \diag v \right], \left[ \mathcal{A} \right]\left[i \mathcal{B} \right]^{-1} \right]+\left[ \left[ \diag a \right], \partial_\mathbf{h} \left[ \mathcal{A} \right]\left[i \mathcal{B} \right]^{-1} \right] = \\ \nonumber
 -&2m^2\left[\left[ \diag v \right], \left[ \diag a \right]^*\right]
 -2m^2 \left[\left[ \diag a \right], \left[ \diag v \right]^*\right] + \frac{1}{2} \left[ \left[\left[ \diag v \right],  \left[{\mathcal A}\right]\left[i{\mathcal B}\right]^{-1} \right],  \left[\left[ \diag a \right],  \left[{\mathcal A}\right]\left[i{\mathcal B}\right]^{-1} \right]\right],
\end{align}
where $\left[ \diag a \right]^*:=\left[ i\mathcal{B}\right] \overline{\left[ i\mathcal{B}\right]}^{-1}\overline{\left[ \diag a \right]\left[ i\mathcal{B} \right]}{\left[ i\mathcal{B}\right]}^{-1}$, etc.

Note that \eqref{eq:Ah} yields easily that $\left[ \left[ \diag a \right], \left[ \mathcal{A} \right]\left[i \mathcal{B} \right]^{-1} \right]$ is rotationally invariant by setting $v = ia$. Recall rotational invariance $\partial_\phi \left[i \mathcal{B}\right] = 0$ was shown directly from rotation of spinors $f_j$.

\subsubsection*{$2$-point case and the Painlev\'e transcendant}

Now we reduce to the $n=2$ case. We may write (recall $\left[ \mathcal{A} \right]\left[i \mathcal{B} \right]^{-1}$ is symmetric thanks to \eqref{eq:gasymp})
\begin{equation*}
\left[i\mathcal{B}\right]=
\begin{pmatrix}
1 & -i\beta \\
i\beta & 1
\end{pmatrix}
;\;
\left[ \left[ \diag a \right], \left[ \mathcal{A} \right]\left[i \mathcal{B} \right]^{-1} \right] =
\begin{pmatrix}
0 &  -\gamma\\
\gamma & 0
\end{pmatrix}.
\end{equation*}
These are rotationally and translationally invariant quantities; therefore, we may write $\beta, \gamma = \beta(r), \gamma(r)$ for $r = |a_1 - a_2|$, and study their derivative with respect to $r$, which corresponds to setting $v = a/r$.

Equations \eqref{eq:Bh} and \eqref{eq:Ah} become
\begin{equation*}
\beta' = \frac{1-\beta^2}{r} \imm  \gamma ,\; \gamma' = 4m^2r\frac{2i\beta(1+\beta^2)}{(1-\beta^2)^2}\mbox{; or }\left(\frac{r\beta'}{1-\beta^2} \right)' = \frac{4m^2r\beta(1+\beta^2)}{(1-\beta^2)^2}.
\end{equation*}
Set $\tanh h := \beta$ such that $h' = \beta'/(1-\beta^2)$, and $\eta := \exp(-2h)$, we finally derive the Painlev\'e III equation
\begin{equation} \label{eq:painlevefinal}
rh'' + h' = {m^2 r}\sinh 4h\mbox{, or }
\eta'' = \eta^{-1}(\eta')^2 - r^{-1}\eta'+m^2 (\eta^3 - \eta^{-1}).
\end{equation}

Notice that $\gamma'$ is purely imaginary and $\gamma$ is zero at infinity by \eqref{eq:A}, so $\gamma$ is purely imaginary. Recovering diagonal elements from \eqref{eq:abdiag}, we have
\begin{equation}\label{eq:betafinal}
\left[ \mathcal{A} \right]\left[i \mathcal{B} \right]^{-1} =
\begin{pmatrix}
 \frac{16m^2\beta^{2}-(\beta')^2}{2(1-\beta^2)^2}(\bar{a_2}-\bar{a_1}) &  \frac{\beta'ri}{(1-\beta^2)(a_2-a_1)}\\
\frac{\beta'ri}{(1-\beta^2)(a_2-a_1)} &  \frac{16m^2\beta^{2}-(\beta')^2}{2(1-\beta^2)^2}(\bar{a_1}-\bar{a_2})
\end{pmatrix}.
\end{equation}

Now we give an estimate at infinity of $\beta$, equivalently that of $\eta$. It must be mentioned that there is an extensive literature on analysis of Painlev\'e transcendents through isomonodromy (see e.g. \cite{fokas}); we include a simple derivation for the benefit of an interested reader since we were not able to find it in Ising-specific isomonodromy literature.

Without loss of generality, assume $a_1 = 0$ and $a_2 = r$. We will calculate $\beta = -\mathcal{B}_{1, 2}$ using \eqref{eq:res}, integrating $f_1$ against $Z^1_{-1/2}(z-a_2)$ along a half-moon shaped contour: $\left(\partial B_D \cap \{ \ree z>0 \}  \right) \cup [Di, -Di]$ for large $D>2r$, approaching the imaginary axis from the right side. Applying \eqref{eq:expo} on $f_1(z)$ and using the explicit identification $Z^1_{-1/2}(re^{i\theta})= \frac{{2\Gamma(1/2)}{|m|^{1/2}}}{\pi} e^{-i\theta/2}K_{1/2}(2|m|r)$ (\eqref{eq:formalpowers}, \cite[(10.27.2)]{dlmf}), we can bound the contribution of the integral along the half circle by $O(D^{1/2}\cdot e^{-5|m|(D-r)/2})$. So we are concerned about
\begin{equation*}
-\frac{1}{2\pi}\ree\int_{[-Di, Di]} f_1(z) Z^1_{-1/2}(z-a_2) dz =O(r e^{-9|m|r/4}+De^{ -9|m|r/4}) -\frac{1}{2\pi}\ree \int_{[-Di, Di]}  Z^1_{-1/2}(z)Z^1_{-1/2}(z-a_2) dz.
\end{equation*}
where we replaced $f_1(z)$ in $[ri/2, -ri/2]$ by $Z^1_{-1/2}(z)$ using \eqref{eq:max} and estimated the rest using \eqref{eq:expo} (we will comment on the choice of sheet later). Notice that none of the functions have a monodromy or a pole within the slit half-moon $\left(\partial B_D \cap \{ \ree z>0 \}  \right) \setminus [r, D]$: so \eqref{eq:gr} gives
\begin{align*}
-\frac{1}{2\pi}\ree \int_{[-Di, Di]}  Z^1_{-1/2}(z)Z^1_{-1/2}&(z-a_2) dz =  O(e^{-4|m|(D-r)})  +\frac{1}{\pi}\ree\int_{[r, D]}  Z^1_{-1/2}(z)Z^1_{-1/2}(z-a_2) dz \\
&=O(e^{-4|m|(D-r)})\pm \frac{1}{\pi}\ree\int_{[r, D]}  \left(\frac{2\Gamma(1/2)|m|^{1/2}}{\pi}\right)^2 K_{1/2}(2|m|x)K_{1/2}(2|m|(x-r)) dz\\
&=O(e^{-4|m|(D-r)})\pm \frac{1}{\pi}\ree \int_r^{\infty} \left(\frac{2\Gamma(1/2)|m|^{1/2}}{\pi}\right)^2 K_{1/2}(2|m|x)K_{1/2}(2|m|(x-r)) dx,
\end{align*}
where two integrals above and below the slit double up. Miraculously, the last integral may be calculated explicitly: $K_{1/2}(x)=\sqrt{\frac{\pi}{2x}}e^{-x}$ \cite[(10.39.2)]{dlmf}, so by \cite[(10.32.8)]{dlmf}
\begin{equation*}
\int_r^{\infty} K_{1/2}(2|m|x)K_{1/2}(2|m|(x-r)) |m|dx=\frac{\pi}{4}\int_r^{\infty} \frac{e^{-4|m|(x-r/2)}}{\sqrt{x(x-r)}} dx=\frac{K_0(2|m|r)}{4}.
\end{equation*}
So (setting $D=2r$), $|\beta| = \frac{ K_0(2|m|r)}{\pi}+O(e^{-17|m|r/8})$. Since we may choose $\beta>0$ ($h>0, \eta<1$), we have
\begin{equation}\label{eq:asympto}
\eta =1- \frac{ 2}{\pi}K_0(2|m|r)+O(e^{-17|m|r/8}).
\end{equation}

\subsection*{Acknowledgements}

The author is supported by a KIAS Individual Grant (MG077202) at Korea Institute for
Advanced Study. This work was initiated during the author's PhD thesis at \'Ecole Polytechnique F\'ed\'erale de Lausanne, where he was supported by ERC grant SG CONSTAMIS, and he thanks the crucial guidance of his advisor Cl\'ement
Hongler. The author also thanks Dmitry Chelkak for inspiring conversations and R\'emy Mahfouf for helpful discussions.

\appendix

\section{Quantitative Estimates}

Here we give estimates of massive holomorphic functions, incorporating the theory of generalized analytic functions (e.g. functions satisfying a \emph{Vekua equation} $\bar{\partial}f = \alpha \bar{f}$ for general functions $\alpha$). The development of this theory goes back to mid-20th century \cite{bers, vek}, and the recently published \cite{baratchart} provides a clear account. While we find this particular theory rather well-suited for our purposes, one could conceivably appeal to alternative forms of elliptic regularity.

\subsubsection*{Bers-Vekua Theory}

The main tool we utilize from Bers-Vekua theory is the following \emph{similarity principle}:
\begin{prop}\label{prop:similarity}
Given a massive holomorphic function or spinor $f$ defined on (possibly punctured) $B_R \subset \mathbb C$, there is a (unique up to a real additive scalar) H\"older continuous function $s$ taking purely real values on $\partial B_R$ such that $e^{-s} f$ is holomorphic at same points where $f$ is massive holomorphic.

If we fix the additive constant so that $\int_{\partial B_R} s =0 $, we have for a universal constant $|s|\leq \const |m|R$.
\end{prop}
\begin{proof}
First, notice that defining $f_R(z) := f(Rz)$, we may work with the fixed domain $B_1$ with mass $mR$. Then the proposition follows from \cite[Lemma 3.1]{baratchart} and Sobolev inequality. Note, since $mR$ is in every $L^p$ space, we get $\alpha$-H\"older regularity for any $\alpha\in (0,1)$.

This result relies on integrating $\overline{f_R}/f_R$ against the Cauchy kernel, so it is valid in the presence of poles or $-1$ monodromies.
\end{proof}

As a simple consequence, we get Proposition \ref{prop:laurent}.
\begin{proof}[Proof of Proposition \ref{prop:laurent}]
by assumption and above, $f = e^{s}\cdot e^{-s}f$ with $e^{-s}f$ being a holomorphic spinor of $O(|z-a|)^\nu$ order around $a$. Therefore there is a holomorphic function $g$ such that $e^{-s(z)}f(z) = (z-a)^\nu g(z)$.  We may set $A_\nu := e^{s(a)}g(a)$ and thanks to the H\"older regularity of $s$ we get $f(z) = A_\nu (z-a)^\nu + o(|z-a|^\nu)$. The second estimate follows then similarly by noticing that the left hand side is a massive holomorphic spinor of $o(|z-a|^\nu)$ order around $a$.
\end{proof}

\subsubsection*{Massive Cauchy Formula and Estimates}

Recall the Cauchy-type residue formula \eqref{eq:res}. Applying the same rationale around a bulk point $a$ yields a Cauchy-type integral formula for $f(a)$. Since the poles of modified Bessel functions of the first kind $I_{-\nu}$ at $0$ degenerates for negative integer indices, one must replace it with the modified Bessel functions of the \emph{second kind} $K_{\nu}$. Namely, one may construct the massive versions $Y^1_{-1},Y^i_{-1}$ of the Cauchy kernel $1/z$; we will give a general formula which works for negative integers or half-integers $-\nu$, which may serve as alternatives to the $Z^1_{-\nu},Z^i_{-\nu}$ so that they decay exponentially at infinity. Define now $X_{-\nu}(z=re^{i\theta})= e^{-i\nu \theta}K_{\nu}(2|m|r)$, and
\begin{equation*}
Y^1_{-\nu}(z) := \frac{2|m|^\nu}{\Gamma(\nu)}\left(X_{-\nu}(z) -(\sgn m)X_{-\nu+1}(z)  \right) ;\; Y^i_{-\nu}(z) := \frac{2|m|^\nu}{\Gamma(\nu)}\left(X_{-\nu}(z)  + (\sgn m) X_{-\nu+1}(z)  \right).
\end{equation*}
Massive holomorphicity of $Y_{-\nu}$ may be verified as in \eqref{eq:formalder} for $Z_\nu$. Replacing both $Z_\nu$ with $Y_{-1}$ in \eqref{eq:res} lets us recover respectively the real and imaginary parts of $f(a)$. Note both $Y_{-1}(z)$ are $O(|z|^{-1})$ and both $\nabla Y_{-1}(z)$ are $O(|z|^{-2})$ at small scales \cite[§10.29, §10.30]{dlmf}, and all are $O(|z|^{-1/2}e^{-2|m||z|})$ at large scales \cite[§10.25]{dlmf}.

At small scales, we have, with universal constants,
\begin{equation}\label{eq:mvt}
|f(0)| \leq  \frac{\const}{R^2} \iint_{B_{R}} |f| \leq \frac{\const}{R} \left[ \iint_{B_{R}} |f|^2\right]^{1/2} \mbox{ and } |\nabla f(0)| \leq \frac{\const }{R^3} \iint_{B_{R}} |f|\leq  \frac{\const}{R^2} \left[ \iint_{B_{R}} |f|^2\right]^{1/2}
\end{equation}
as long as $f$ is massive holomorphic in $B_{R}$ for $R<1$. This estimate in terms of area integrals come from, say, radially integrating the Cauchy-type integral on concentric circles.

Now we give estimates we need specifically for the Ising spinors: recall $m<0$ is assumed.

\begin{lem}
Suppose the ball $B_{R}(a_j)$, $R<1$, contains no other $a_{k'}$. Then
\begin{align}\label{eq:Aunif}
\left| f_k(z) -  \left( i\mathcal{B}_{j,k} \right) \bullet Z_{-1/2}(z-a_j) \right| \leq \const |z-a_j|^{1/2}\mbox{ in }z\in B_{R}(a_j)
\end{align}
for constant only dependent on $m$ and $R$.
\end{lem}

\begin{proof}
Notice $f_k -  \left( i\mathcal{B}_{j,k} \right) \bullet Z_{-1/2}(z-a_j) $ is $L^2$ bounded in $B_{R}(a_j)$ by constant only depending on $m$: indeed, both $f_k$ and $ \left( i\mathcal{B}_{j,k} \right) \bullet Z_{-1/2}(z-a_j) $ are (recall $\left|\mathcal{B}_{j,k}\right| \leq 1$). Then applying Proposition \ref{prop:similarity}, we see that
\begin{align*}
f_k -  \left( i\mathcal{B}_{j,k} \right) \bullet Z_{-1/2}(z-a_j)  = e^{s}g,
\end{align*}
where $g$ is a holomorphic spinor in $B_{R}(a_j)$ with asymptotic $O(|z-a_j|^{1/2})$ around $a_j$ and $|s|\leq \const mR$. Therefore, $g^2$ is a holomorphic function vanishing at $a_j$ and whose $L^1$ norm is bounded by a constant depending on $m, R$. Then by Cauchy's formula or otherwise, \eqref{eq:Aunif} follows.
\end{proof}
 
At large scales, we have exponential dropoff. If $B_R$ for $R>1$ contains no $a_j$'s, then for constants only depending on $m$
\begin{equation}\label{eq:expo}
|f_k(0)| \leq  \frac{\const e^{-3|m|R/4}}{R^{1/2}} \iint_{B_{R}} |f_k| \leq {\const e^{-3|m|R/4}}{R^{1/2}} \left[ \iint_{B_{R}} |f|^2\right]^{1/2} \leq {\const e^{-|m|R/2}}.
\end{equation}

It is useful to recall that massive holomorphic functions still satisfy maximum modulus principle: indeed, one may verify $\Delta |f|^2=4\partial\bar\partial |f|^2 = 8m^2|f|^2 + 2|\nabla\ree f|^2+2|\nabla \imm f|^2\geq 0$, so there cannot be a strict minimum of $|f|$ in the bulk. Therefore, if $B_R(a_k)$, $R>1$, contains no other $a_{j'}$, we still have
\begin{equation}\label{eq:max}
\left| f_k(z) -   Z^1_{-1/2}(z-a_k) \right| \leq  {\const e^{-|m|R/4}}\mbox{ in }z\in B_{R/2}(a_k),
\end{equation}
since $Z^1_{-1/2}(re^{i\theta})= \frac{{2\Gamma(1/2)}{|m|^{1/2}}}{\pi} e^{-i\theta/2}K_{1/2}(2|m|r)$ satisfies the same bound on $\partial B_{R/2}$ and the maximum modulus principle applies inside.

If $B_R(a_j)$, $R>1$, for general $j$ contains no other $a_{j'}$, similarly
\begin{align}\label{eq:A}
|\mathcal{A}_{j,k}| \leq {\const R^{-1/2}e^{-5|m|R/4}},
\end{align}
estimating $\mathcal{A}_{j,k}$ through \eqref{eq:res} but instead integrating against $Y^1_{-3/2}, Y^i_{-3/2}$ (recall both are $O(|z|^{-1/2}e^{-2|m||z|})$), on $\partial B_{R/2}(a_j)$.

\end{document}